\documentclass[11pt]{article}
\usepackage{latexsym,bm}
\usepackage{mathrsfs}
\usepackage{amsmath,amssymb,amsthm,mathtools}
\usepackage{enumitem}
\usepackage{amsmath,amsthm}
\usepackage{graphicx}
\usepackage{amssymb}
\usepackage{CJK}
\usepackage{color}
\usepackage{cite}
\usepackage{comment}
\usepackage{array}
\usepackage{stfloats}
\usepackage[colorlinks,
            linkcolor=blue,       
            anchorcolor=blue,  
            citecolor=blue,        
            ]{hyperref}
\usepackage{caption}
\usepackage{graphicx}
\usepackage{epstopdf}
\usepackage{upgreek}
\usepackage{subcaption}

\newtheorem{theorem}{Theorem}[section]
\newtheorem{lemma}[theorem]{Lemma}

\newtheorem{corollary}[theorem]{Corollary}
\newtheorem{problem}[theorem]{Problem}

\theoremstyle{remark}
\newtheorem{remark}[theorem]{Remark}
\theoremstyle{definition}
\newtheorem{definition}[theorem]{Definition}

\DeclareGraphicsExtensions{.eps,.eps.gz}
\usepackage[top=2.5cm,bottom=2.5cm,left=2.8cm,right=2.2cm]{geometry}
\normalsize \rm
\allowdisplaybreaks[4]
\makeatletter
\@addtoreset{equation}{section}
\makeatother

 \usepackage{indentfirst}
\begin{document}
\begin{CJK}{GBK}{song}
\newcommand{\song}{\CJKfamily{song}}    
\newcommand{\fs}{\CJKfamily{fs}}        
\newcommand{\kai}{\CJKfamily{kai}}      
\newcommand{\hei}{\CJKfamily{hei}}      
\newcommand{\li}{\CJKfamily{li}}        
\renewcommand\figurename{Fig.}

\begin{center}
{{\huge Cop numbers for subclasses of partial cubes}} \\[18pt]
{\Large Zhaoman Huang$^{1}$, Yan-Ting Xie$^{1}$, \footnotetext{*Corresponding author.\\ E-mail address: shjxu@lzu.edu.cn (S.-J. Xu).} Shou-Jun Xu$^{1,*}$ }\\[6pt]
{ \footnotesize  $^{1}$ School of Mathematics and Statistics, Gansu Center for Applied Mathematics, Lanzhou University, Lanzhou, Gansu 730000, China}
\end{center}
\vspace{1mm}
\begin{abstract}
The game of Cops and Robbers is a classical pursuit--evasion game on graphs. For a graph $G$, the cop number $c(G)$ is the minimum number of cops needed to guarantee the capture of a robber on $G$. Although this parameter has been determined for several fundamental graph classes, comparatively few exact results are known for partial cubes and their subclasses. We first establish an upper bound for every finite median graph \(M\) in terms of its tree-dimension, which improves Crawford and Ir\v{s}i\v{c} Chenoweth's bound significantly. This result refines the previous upper bound expressed in terms of a hypercube embedding dimension and can give a substantially smaller estimate.  Then we investigate the cop numbers of simplex graphs---a subclass of partial cubes. For a finite graph \(G\), the simplex graph \(S(G)\) has the cliques of \(G\), including the empty clique, as its vertices, with two cliques adjacent whenever they differ in exactly one vertex. We establish a general lower bound for \(c(S(G))\) in terms of the clique number of \(G\) and a general upper bound in terms of its chromatic number. 
Finally, as direct applications, we determine the exact values of cop numbers of some special simplex graphs---bipartite wheels, Fibonacci and Lucas cubes. 

\end{abstract}

\section{Introduction}

The game of Cops and Robbers is a classical pursuit-evasion game on graphs. In its standard version, the game is played on a finite connected graph $G$ according the following rules: a set of cops first chooses initial vertices (two or more cops choosing the same initial vertex is allowed), and then the robber chooses an initial vertex. Thereafter the players move alternately, with the cops moving first. At each turn, every player may either remain at the current vertex or move to an adjacent vertex. The cops win if, at some time, one of them occupies
the same vertex as the robber. The robber wins if he can evade
capture indefinitely. The \emph{cop number} of $G$, denoted by $c(G)$, is the minimum number of cops that guarantee capture of the robber. Cops and Robbers has even found application in robotics, artificial intelligence, and so called moving target search; see \cite{IsazaLuBulitkoGreiner2008,MoldenhauerSturtevant2009}.

Since its introduction by Quilliot and by Nowakowski and Winkler, the cop number has been studied from structural, algorithmic, and extremal perspectives \cite{BonatoNowakowski2011,Bradshaw2023, Chudnovsky2024,Lehner2021,NowakowskiWinkler1983,Quilliot1983}. Fundamental results include the characterization of cop-win graphs as dismantlable graphs, Aigner and Fromme's \cite{AignerFromme1984} theorem that every planar graph has cop number at most three, Schr\"{o}der's \cite{Schroeder1983} extension of this bound to graphs of bounded genus and showing that every graph of genus $g$ has cop number at most $\lfloor 3g/2\rfloor+3$, and the formula of Maamoun and Meyniel \cite{MaamounMeyniel1987} for the cop number of Cartesian products of nontrivial trees. In particular, their result determines the cop number of every hypercube.

For general graphs, the most prominent open problem is Meyniel's conjecture. It asserts that there exists a constant $C$ such that every connected graph $G$ on $n$ vertices satisfies $c(G) \le C\sqrt{n}.$ This order of magnitude would be best possible, since incidence graphs of finite projective planes and related high-girth constructions give lower bounds of order $\Omega(\sqrt n)$. Meyniel's conjecture remains open and is widely regarded as one of the central problems in the theory of Cops and Robbers \cite{BairdBonato2012,BonatoNowakowski2011}. Nevertheless, it has been verified for several families; for instance, it has been established for random $d$-regular graphs \cite{Pralat2019}, binomial random graphs \cite{PralatWormald2015}, graphs of diameter 2 \cite{LuPeng2012}, bipartite graphs of diameter 3 \cite{LuPeng2012}, and abelian Cayley graphs \cite{Bradshaw2020}.

The present paper focuses on the cop number of partial cubes and their subclasses. Recall that a partial cube is a graph that admits an isometric embedding into a hypercube $Q_{n}$. Partial cubes form a broad and structurally rich class that includes trees, hypercubes, median graphs, Fibonacci cubes, Lucas cubes, and numerous graph families arising in interconnection networks and chemical graph theory. Despite the extensive structural theory developed for partial cubes, comparatively little is known about their cop numbers. Crawford and Ir\v{s}i\v{c} Chenoweth \cite{CrawfordIrsicChenoweth2025} recently initiated the systematic study of this problem. They established general lower bounds for partial cubes, obtained upper bounds for important subclasses such as median graphs, and stated bounds for Fibonacci and Lucas cubes. Their work motivates the problem of determining the exact cop numbers for particular, structurally meaningful subclasses of partial cubes.

Crawford and Ir\v{s}i\v{c} Chenoweth \cite{CrawfordIrsicChenoweth2025} showed that if a finite median graph \(M\) (i.e., a partial cube in which every triple of vertices has a unique median) admits an isometric embedding into \(Q_d\), then $c(M)\leq \left\lceil(d+1)/2\right\rceil$. We refine this result by replacing the dimension of the ambient hypercube with the tree-dimension of \(M\). More precisely, we prove that for any finite median graph $M$,
\[
c(M)\leq \left\lceil (tdim(M)+1)/2 \right\rceil = \lceil (\chi(M^\#)+1)/2 \rceil,
\]
where $tdim(M)$ denotes the tree-dimension of $M$, $\chi(M^\#)$ denotes the chromatic number of the crossing graph $M^\#$ of $M$. Since \(Q_d\) is the Cartesian product of \(d\) copies of \(K_2\), every isometric embedding of \(M\) into \(Q_d\) gives a representation of \(M\) inside a product of \(d\) trees. Therefore, $tdim(M)\leq d$ and our estimate is never weaker than the previous hypercube-dimension bound.

Simplex graphs offer a framework for establishing exact cop numbers for subclasses of partial cubes. For a finite graph \(G\), the simplex graph \(S(G)\) is the graph whose vertices are the cliques of \(G\), including the empty clique, with two cliques adjacent whenever they differ in exactly one vertex. Simplex graphs were introduced by Bandelt and van de Vel \cite{BandeltVanDeVel1989}, who proved that every simplex graph is a median graph and hence a partial cube. Subsequently, Klav\v{z}ar and Mulder
\cite{KlavzarMulder2002} proved that the crossing graph of \(S(G)\) is isomorphic to \(G\), that is, $(S(G))^{\#}\cong G$. This identity makes it possible to relate the cop number of \(S(G)\) to structural parameters of \(G\).

Several familiar families of partial cubes arise as simplex graphs. Since every subset of the vertex set of the complete graph \(K_n\) is a clique, one has $S(K_n)\cong Q_n$. Cycles provide another basic example. For \(n\geq3\), let \(BW_n\) denote the bipartite wheel, also called the gear graph, obtained from the wheel with rim \(C_n\) by subdividing every rim edge exactly once \cite{Kirlangic2009}. Under the representation of the simplex graph, the empty clique of \(C_n\) corresponds to the hub, the singleton cliques correspond to the original rim vertices, and the edge cliques correspond to the subdivision vertices. Therefore, $S(C_n)\cong BW_n$ \cite{XieXu2025}. Gear graphs have also been studied in graph-embedding and network-vulnerability settings, including wirelength, rupture degree, integrity, and domination integrity \cite{GreeniJoshwa2023,Kirlangic2009,SundareswaranSwaminathan2016}.

Simplex graphs also provide representations of Fibonacci and Lucas cubes. The \(n\)-dimensional Fibonacci cube \(\Gamma_n\) is the subgraph of \(Q_n\) induced by the binary strings of length \(n\) containing no two consecutive \(1\)'s. The \(n\)-dimensional Lucas cube \(\Lambda_n\) is its cyclic analogue, obtained by imposing the additional condition that the first and last coordinates are not both equal to \(1\). Since cliques of the complement of a graph are precisely its independent sets, these definitions yield $\Gamma_n\cong S(\overline{P_n})$ and $\Lambda_n\cong S(\overline{C_n})$, where the latter representation is used for \(n\geq3\) \cite{Klavzar2013,MunariniCippoSalvi2001,Taranenko2013}.

Fibonacci and Lucas cubes are important not only as prominent subclasses of partial cubes but also because of their applications. Fibonacci cubes were introduced by Hsu \cite{Hsu1993} as interconnection topologies for parallel and distributed computing systems, and their structural, recursive, and enumerative properties have since been extensively studied. A comprehensive treatment of Fibonacci cubes, Lucas cubes, and related graph families can be found in the recent monograph of E\u{g}ecio\u{g}lu, Klav\v{z}ar, and Mollard \cite{EgeciogluKlavzarMollard2023}. These graph families also arise in chemical graph theory. In particular, Fibonacci cubes occur as resonance graphs of fibonaccenes, where vertices represent perfect matchings and adjacency corresponds to a face rotation \cite{Klavzar2005}. Lucas cubes and related constructions likewise appear in the study of resonance graphs associated with cyclic molecular systems, including cyclic polypyrenes and armchair nanotubes \cite{Zigert2013}.

These representations motivate us to study the cop number of \(S(G)\) by structural parameters of the base graph \(G\). One of our main results gives general lower and upper bounds for \(c(S(G))\) in terms of the clique number and chromatic number of \(G\). In particular, these bounds coincide when \(G\) is perfect, yielding an exact formula for the cop number of \(S(G)\).

The general bounds have several immediate applications. The representation \(Q_n\cong S(K_n)\) recovers the classical result for the cop number of hypercubes. Similarly, the same bounds determine the exact cop number of every bipartite wheel by the representation \(S(C_n)\cong\mathrm{BW}_n\) for \(n\geq3\). We also apply the lower and upper bounds for simplex graphs to the complements of paths $\overline{P_n}$ and cycles $\overline{C_n}$. From the representation $\Gamma_n\cong S(\overline{P_n})$, we determine the cop number of every Fibonacci cube exactly for every \(n\geq0\). For Lucas cubes, the representation $\Lambda_n\cong S(\overline{C_n})$ for \(n\geq3\) shows that the general bounds coincide when \(n\equiv0,1,2\pmod4\). In the remaining case \(n\equiv3\pmod4\), we develop a block-projection strategy that closes the one-cop gap between the general lower and upper bounds.

The remainder of this paper is organized as follows. We begin in Section~2 with the necessary definitions and preliminary results. Section~3 develops refined upper bounds for the cop number of median graphs and general lower and upper bounds for the cop numbers of simplex graphs. Applications to bipartite wheels and Fibonacci cubes are presented next in Section~4, followed by the treatment of Lucas cubes in Section~5. Concluding remarks are given in Section~6.

\section{Preliminaries}

Throughout this paper, all graphs are finite, simple, and undirected. Graphs on which the game of Cops and Robbers is played are always assumed to be connected. For a graph \(G\), its \emph{vertex set} and \emph{edge set} are denoted by \(V(G)\) and \(E(G)\), respectively. The \emph{complement} of \(G\) is denoted by \(\overline{G}\). For a subset \(S\subseteq V(G)\), the \emph{subgraph} of \(G\) induced by \(S\) is denoted by \(G[S]\). For sets \(A\) and \(B\), we write \(A\setminus B=\{x\in A:x\notin B\}\) for their set difference and
\(A\mathbin{\triangle}B=(A\setminus B)\cup(B\setminus A)\) for their symmetric difference.

A \emph{clique} of $G$ is a set of pairwise adjacent vertices. The maximum cardinality of a clique of $G$ is called the \emph{clique number} of $G$ and is denoted by $\omega(G)$.
An \emph{independent set} of $G$ is a set of pairwise non-adjacent vertices. The maximum cardinality of an independent set of $G$ is called the \emph{independence number} of $G$ and is denoted by $\alpha(G)$. A graph is \emph{complete} if its vertex set is a clique. The complete graph on $n$ vertices is denoted by $K_n$. 

For a graph $G$, a \emph{proper $k$-coloring} of $G$ is a map $\varphi:V(G)\to \{1,2,\ldots,k\}$ such that adjacent vertices receive distinct colors. The \emph{chromatic number} of $G$, denoted by $\chi(G)$, is the least integer $k$ for which $G$ admits a proper $k$-coloring. We use the convention that the chromatic number and the clique number of the empty graph are both zero. A graph \(G\) is \emph{perfect} if every induced subgraph \(H\) of \(G\) satisfies \(\chi(H)=\omega(H)\).

The \emph{path} on \(n\) vertices is denoted by $P_n=\langle v_1,v_2,\ldots,v_n\rangle$, where
$E(P_n)=\{v_iv_{i+1}:1\leq i\leq n-1\}$. More generally, a path $P=\langle v_0,v_1,\ldots,v_k\rangle$ is
called a \emph{$v_{0},v_{k}$-path}. The \emph{cycle} on \(n\) vertices, where \(n\geq3\), is denoted by $C_n=\langle u_1,u_2,\ldots,u_n,u_1\rangle$, where $E(C_n)=\{u_iu_{i+1}:1\leq i\leq n-1\}\cup\{u_nu_1\}$. The \emph{length} of a path or a cycle is its number of edges. Thus, \(P_n\) has length \(n-1\), whereas \(C_n\) has length \(n\). For vertices \(x,y\in V(G)\), the \emph{distance} \(d_G(x,y)\) is the length of a shortest \(x,y\)-path in \(G\), and such a path is called an \emph{\(x,y\)-geodesic}. For two vertices $u$ and $v$ of $G$, the \emph{interval} $I(u,v)$ between $u$ and $v$ in $G$ is the set of all vertices lying on some shortest $u,v$-path. A subgraph \(H\) of \(G\) is \emph{isometric} if $d_H(x,y)=d_G(x,y)$ for all \(x,y\in V(H)\), and it is \emph{convex} if every \(x,y\)-geodesic in \(G\) is contained in \(H\) for all \(x,y\in V(H)\). For a graph $G$, it is obvious that its convex subgraphs are isometric and its isometric ones are induced and connected. The \emph{Cartesian product} \(G\square H\) of two graphs \(G\) and \(H\) has vertex set \(V(G)\times V(H)\), where two vertices \((g,h)\) and \((g',h')\) are adjacent if either $gg'\in E(G)~\text{and}~h=h'$, or $g=g'~\text{and}~hh'\in E(H).$

The \emph{$n$-dimensional hypercube}, denoted by $Q_n$, is the graph whose vertex set consists of all binary strings of length $n$, i.e., $V(Q_n) = \{0, 1\}^n$, where two vertices are adjacent if and only if their strings differ in exactly one bit position. A graph $G$ is a \emph{partial cube} if it is an isometric subgraph of a hypercube.

The \emph{Djokovi\'{c}-Winkler relation} $\Theta$ (see \cite{Djokovic1973, Winkler1984}) is defined on $E(G)$ by: for edges $e=uv$ and $f=xy$,
\[
e\,\Theta\,f \iff d_G(u,x)+d_G(v,y) \neq d_G(u,y)+d_G(v,x).
\]
The relation $\Theta$ is clearly reflexive and symmetric. If $G$ is bipartite, this is equivalent to $d_G(u,x)=d_G(v,y)$ and $d_G(u,y)=d_G(v,x)$. Winkler \cite{Winkler1984} proved that a connected graph $G$ is a partial cube if and only if it is bipartite and $\Theta$ is an equivalence relation on $E(G)$. The Djokovi\'{c}--Winkler relation $\Theta$ partitions the edges of a partial cube into equivalence classes, called \emph{$\Theta$-classes}. In a partial cube $G$, the number of $\Theta$-classes equals the {\em isometric dimension} of $G$, denoted by $idim(G)$, defined as the smallest integer $n$ for which $G$ admits an isometric embedding into the $n$-dimensional hypercube $Q_n$. Moreover, each $\Theta$-class corresponds bijectively to a coordinate position in the binary labels assigned to the vertices of $G$.

For a $\Theta$-class $E$ of a partial cube $G$, deleting all edges of $E$ separates $G$ into two convex components, called the two \emph{semicubes} associated with $E$. We denote their vertex sets by $W_E^0$ and $W_E^1$. Equivalently, if \(xy\in E\), then, after possibly interchanging the labels \(0\) and \(1\),
\[
W_E^0=\{a\in V(G): d_G(x,a)<d_G(y,a)\},~\text{and}~
W_E^1=\{a\in V(G): d_G(y,a)<d_G(x,a)\}.
\]

The crossing graph was studied systematically by Klav\v{z}ar and Mulder \cite{KlavzarMulder2002}.

\begin{definition}\cite{KlavzarMulder2002}\label{def:crossing}
Let $G$ be a partial cube. Two $\Theta$-classes $E$ and $F$ are said to \emph{cross} if
\[
  W_E^\varepsilon\cap W_F^\delta\ne\emptyset
  ~\text{for all }\varepsilon,\delta\in\{0,1\}.
\]
The \emph{crossing graph} $G^{\#}$ of $G$ is the graph whose vertices are the $\Theta$-classes of $G$, two vertices being adjacent if and only if the corresponding $\Theta$-classes cross. In particular, the crossing graph of $K_1$ is empty graph.
\end{definition}

A \emph{direction} is a set of pairwise non-crossing $\Theta$-classes.

Median graphs, originally studied by Avann in the language of metric ternary distributive semi-lattices~\cite{Avann1961}, were later investigated systematically in graph-theoretic form by Bandelt and Barth\'elemy \cite{BandeltBarthelemy1984}. Recall that every median graph is a partial cube \cite{Ovchinnikov2011}.

\begin{definition}\cite{BandeltBarthelemy1984}\label{median}
A graph $M$ is called a \emph{median graph} if, for every triple $x,y,z\in V(M)$, there exists a unique median vertex $m=m_{M}(x,y,z)\in V(G)$ satisfying that $m\in I(x,y)\cap I(x,z)\cap I(y,z)$.
\end{definition}

Trees and hypercubes are median graphs, and the Cartesian product of a finite number of median graphs is also a median graph (see \cite{HammackImrichKlavzar2011}).

 An induced subgraph \(H\) of a median graph \(M\) is called a \emph{median subgraph} of \(M\) if, for all \(x,y,z\in V(H)\), the median \(m_M(x,y,z)\) belongs to \(V(H)\); equivalently, the median operation of \(M\) restricts to the median operation of \(H\).

Simplex graphs were introduced by Bandelt and van de Vel
\cite{BandeltVanDeVel1989}.

\begin{definition}\cite{BandeltVanDeVel1989}\label{def:simplex-graph}
Let \(G\) be a finite graph. The \emph{simplex graph} of \(G\), denoted by \(S(G)\), is the graph whose vertices are the cliques of \(G\), including the empty clique, with two cliques \(A\) and \(B\) adjacent if and only if $|A\mathbin{\triangle}B|=1.$
\end{definition}

After identifying each subset of \(V(G)\) with its
characteristic vector in \(\{0,1\}^{V(G)}\), the simplex graph \(S(G)\) is the subgraph of the hypercube induced by the characteristic vectors of the cliques of \(G\). We shall use the following fundamental property of simplex graphs.

\begin{lemma}\label{lem:simplex-median}\cite{BandeltVanDeVel1989}
For every finite graph \(G\), the simplex graph \(S(G)\) is an
isometric subgraph of \(Q_{|V(G)|}\) and is a median graph.
\end{lemma}

We next introduce the bipartite wheel, also known as the gear graph, which is related to the simplex graph of a cycle.

\begin{definition}\label{def:bipartite-wheel}
Let \(n\geq3\), and let \(W_{n+1}\) be the wheel consisting of a cycle $C_n=\langle u_1,u_2,\ldots,u_n,\\u_1\rangle$ and a hub \(h\) adjacent to every vertex of \(C_n\). The \emph{bipartite wheel} \(BW_n\) is obtained from \(W_{n+1}\)
by subdividing every edge of its rim exactly once.
\end{definition}

Equivalently, writing \(v_i\) for the subdivision vertex on the edge \(u_i u_{i+1}\), where indices are read modulo \(n\), we have
\[
V(BW_n)
=
\{h\}\cup\{u_i,v_i:1\leq i\leq n\}
\]
and
\[
E(BW_n)
=
\{hu_i,u_iv_i,v_iu_{i+1}:1\leq i\leq n\}.
\]
In particular, \(BW_n\) is bipartite, with bipartition
$\{u_1,u_2,\ldots,u_n\}$ and $\{h,v_1,v_2,\ldots,v_n\}$. It has \(2n+1\) vertices and \(3n\) edges. The bipartite wheels \(BW_n\) for $3\le n\le 6$ are shown in Fig. \ref{F3}.

\begin{figure}[htbp]
\centering
\subcaptionbox{$BW_{3}$\label{subfig:a}}{%
    \includegraphics[width=0.25\textwidth]{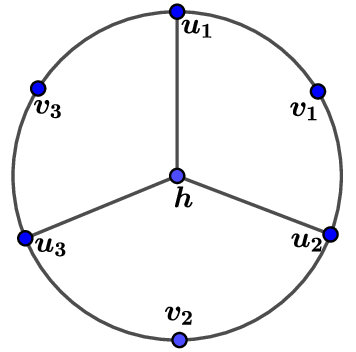}%
}%
\hfill
\subcaptionbox{$BW_{4}$\label{subfig:a}}{%
    \includegraphics[width=0.25\textwidth]{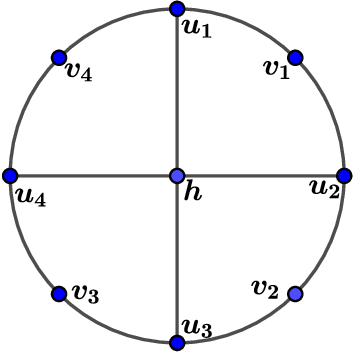}%
}%
\hfill
\subcaptionbox{$BW_{5}$\label{subfig:b}}{%
    \includegraphics[width=0.25\textwidth]{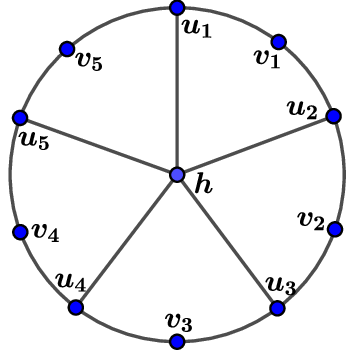}%
}%
\hfill
\subcaptionbox{$BW_{6}$\label{subfig:c}}{%
    \includegraphics[width=0.25\textwidth]{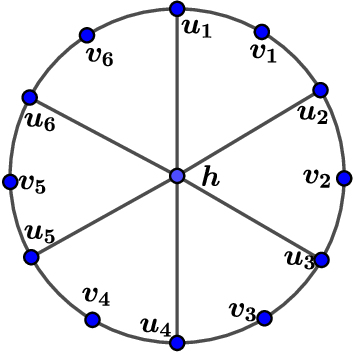}%
}%
\hfill
\caption{The bipartite wheel \(BW_n\) for $3\le n\le 6$.}
\label{F3}
\end{figure}

We next define two further subclasses of partial cubes that arise from binary strings with forbidden consecutive coordinates---the Fibonacci and Lucas cubes. The Fibonacci cube was introduced by Hsu as an interconnection topology for parallel and distributed systems~\cite{Hsu1993}.

\begin{definition}\cite{Hsu1993}
The \emph{$n$-dimensional Fibonacci cube}, denoted by $\Gamma_n$, is the subgraph of the $n$-dimensional hypercube $Q_n$ induced by all binary strings of length $n$ that do not contain two consecutive $1$'s. Formally,
\[
V(\Gamma_n)
=
\{x\in \{0,1\}^n : x \text{ does not contain the substring } ``11"\},
\]
where two vertices are adjacent if their corresponding strings differ in exactly one bit.
\end{definition}
The $n$-dimensional Fibonacci cubes $\Gamma_{n}$ for $0\le n\le 5$ are shown in Fig. \ref{F1}.

\begin{figure}[htbp]
\centering

\subcaptionbox{$\Gamma_{0}$\label{subfig:a}}{%
    \includegraphics[width=0.15\textwidth]{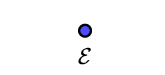}%
}%
\hfill
\subcaptionbox{$\Gamma_{1}$\label{subfig:b}}{%
    \includegraphics[width=0.12\textwidth]{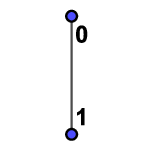}%
}%
\hfill
\subcaptionbox{$\Gamma_{2}$\label{subfig:c}}{%
    \includegraphics[width=0.15\textwidth]{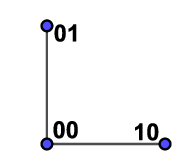}%
}%
\par 

\hfill
\subcaptionbox{$\Gamma_{3}$\label{subfig:d}}{%
    \includegraphics[width=0.25\textwidth]{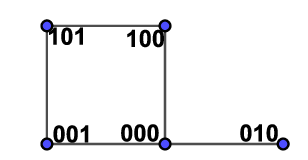}%
}%
\hfill
\subcaptionbox{$\Gamma_{4}$\label{subfig:e}}{%
    \includegraphics[width=0.25\textwidth]{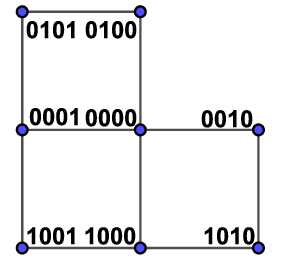}%
}%
\hfill
\hfill
\subcaptionbox{$\Gamma_{5}$\label{subfig:d}}{%
    \includegraphics[width=0.3\textwidth]{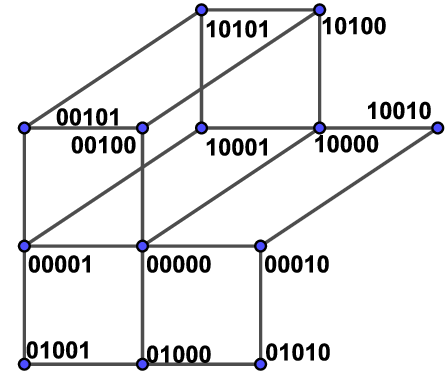}%
}%
\hfill
\caption{The $n$-dimensional Fibonacci cubes $\Gamma_{n}$ for $0\le n\le 5$.}
\label{F1}
\end{figure}

Lucas cubes were introduced by Munarini et al.~\cite{MunariniCippoSalvi2001} as cyclic analogues of Fibonacci cubes.

\begin{definition}\cite{MunariniCippoSalvi2001}
For $n\ge2$, the \emph{$n$-dimensional Lucas cube}, denoted by $\Lambda_n$, is the subgraph of the $n$-dimensional hypercube $Q_n$ induced by all binary strings
$x=x_1x_2\cdots x_n\in\{0,1\}^n$ satisfying:
\begin{enumerate}[label=(\roman*)]
 \item $x$ does not contain the substring ``11'';
    \item $x_1=1$ and $x_n=1$ do not both hold, i.e., $x$ does not start
    and end with $1$.
\end{enumerate}
Vertices are adjacent if their strings differ in exactly one bit.
\end{definition}
We adopt the conventional definitions $\Lambda_0=\Lambda_1=K_1.$ The $n$-dimensional Lucas cubes $\Lambda_{n}$ for $0\le n\le 5$ are shown in Fig. \ref{F2}.
\begin{figure}[htbp]
\centering

\subcaptionbox{$\Lambda_{0}$\label{subfig:a}}{%
    \includegraphics[width=0.15\textwidth]{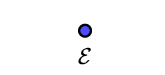}%
}%
\hfill
\subcaptionbox{$\Lambda_{1}$\label{subfig:b}}{%
    \includegraphics[width=0.15\textwidth]{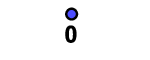}%
}%
\hfill
\subcaptionbox{$\Lambda_{2}$\label{subfig:c}}{%
    \includegraphics[width=0.15\textwidth]{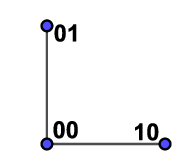}%
}%
\par 

\hfill
\subcaptionbox{$\Lambda_{3}$\label{subfig:d}}{%
    \includegraphics[width=0.25\textwidth]{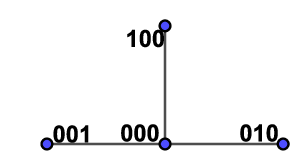}%
}%
\hfill
\subcaptionbox{$\Lambda_{4}$\label{subfig:e}}{%
    \includegraphics[width=0.25\textwidth]{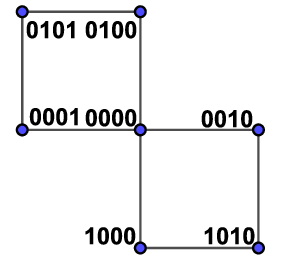}%
}%
\hfill
\hfill
\subcaptionbox{$\Lambda_{5}$\label{subfig:d}}{%
    \includegraphics[width=0.3\textwidth]{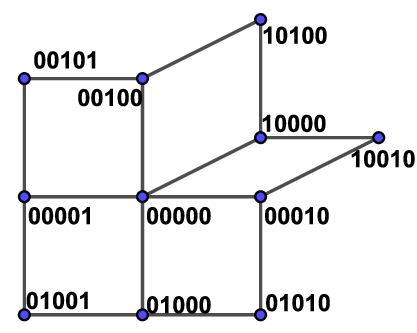}%
}%
\hfill
\caption{The $n$-dimensional Lucas cubes $\Lambda_{n}$ for $0\le n\le 5$.}
\label{F2}
\end{figure}

We use the standard edge-preserving notion of graph retraction.

\begin{definition}\label{def:retraction}\cite{Bandelt1984}
Let \(H\) be an induced subgraph of a graph \(G\). A map $r\colon V(G)\longrightarrow V(H)$ is called a \emph{retraction} if $r(h)=h$ for every $h\in V(H)$, and $r(u)r(v)\in E(H)$ whenever $uv\in E(G)$. If such a map exists, then \(H\) is called a \emph{retract} of \(G\).
\end{definition}

Berarducci and Intrigila \cite{BerarducciIntrigila1993} proved the following monotonicity result for non-expansive retractions, under which an edge may be mapped either to an edge or to a single vertex. It therefore applies, in particular, to the edge-preserving retractions adopted in Definition~\ref{def:retraction}.

\begin{lemma}\cite{BerarducciIntrigila1993}\label{lem:retract-cop}
If $H$ is a retract of $G$, then
\[
  c(H)\le c(G).
\]
\end{lemma}


The following theorem of Maamoun and Meyniel \cite{MaamounMeyniel1987} will be useful to our proofs.

\begin{lemma}\cite{MaamounMeyniel1987}\label{thm:MM}
If $T_1,\ldots,T_t$ are finite nontrivial trees, then
\[
  c(T_1\square\cdots\square T_t)=\lceil(t+1)/2\rceil.
\]
In particular,
\[c(Q_{d})=\lceil(d+1)/2\rceil.\]
We use the conventions \(Q_0=K_1\) and \(c(K_1)=1\).
\end{lemma}

\section{Bounds on the cop number of median graphs and simplex graphs}

In this section, 
we prove an upper bound for median graphs in terms of their tree-dimension, which improves Crawford and Ir\v{s}i\v{c} Chenoweth's bound \cite{CrawfordIrsicChenoweth2025}. Subsequently, we apply this result to simplex graphs to derive general upper bounds for their cop numbers.

The notion of tree-dimension of a finite graph and the fundamental theorem of Bandelt and van de Vel \cite{BandeltVanDeVel1989} are presented below.

\begin{definition}\cite{Eppstein2005}
The \emph{tree-dimension} of a finite graph $G$, denoted by $tdim(G)$, is the minimum nonnegative integer $t$ such that $G$ admits an isometric embedding into a Cartesian product of $t$ trees.
\end{definition}

Bandelt and van de Vel \cite{BandeltVanDeVel1989} characterized the tree-dimension of a finite median graph in terms of directions of its \(\Theta\)-classes. Building upon their result, we extend this characterization to arbitrary partial cubes by determining their tree-dimension in what follows. 

\begin{theorem}\label{thm:tree-dimension-crossing}
Let \(G\) be a finite partial cube and $G\neq K_1$. Then
\[
tdim(G)=\chi(G^\#).
\]
\end{theorem}

\begin{proof}

Set \(t=tdim(G)\). There exist trees \(T_1,\ldots,T_t\) and an isometric embedding
\[
\varphi\colon G\longrightarrow T_1\square\cdots\square T_t.
\]
Each \(\Theta\)-class of \(G\) is mapped to a \(\Theta\)-class belonging to exactly one tree factor. Two \(\Theta\)-classes assigned to the same factor cannot cross, since the \(\Theta\)-classes of a tree are pairwise non-crossing. Thus, assigning to each \(\Theta\)-class the index of its tree factor gives a proper \(t\)-coloring of \(G^\#\). Therefore, $\chi(G^\#)\leq t=tdim(G)$.

Conversely, let \(k=\chi(G^\#)\), and let $\mathcal{E}_1,\ldots,\mathcal{E}_k$ be the color classes of a proper \(k\)-coloring of \(G^\#\). The \(\Theta\)-classes in each \(\mathcal{E}_i\) are pairwise non-crossing.

For each \(i\), contract all edges in \(\Theta\)-classes of \(G\) that do not belong to \(\mathcal{E}_i\), and denote the resulting quotient graph by \(T_i\). By the standard contraction property of partial cubes, \(T_i\) is a partial cube whose \(\Theta\)-classes are precisely the members of \(\mathcal{E}_i\). Since these classes are pairwise non-crossing, \(T_i\) is a tree. In fact, a shortest cycle in \(T_i\) would be isometric and would contain two crossing \(\Theta\)-classes.

For \(x\in V(G)\), let \(x_i\) denote its image in \(T_i\), and define $\pi(x)=(x_1,\ldots,x_k)$. For all \(x,y\in V(G)\),
\[
d_{T_i}(x_i,y_i)
=
\bigl|
\{E\in\mathcal{E}_i:
E\text{ separates }x\text{ and }y\}
\bigr|.
\]
Therefore,
\begin{align*}
d_{T_1\square\cdots\square T_k}
\bigl(\pi(x),\pi(y)\bigr)
&=
\sum_{i=1}^{k}d_{T_i}(x_i,y_i)\\
&=
\bigl|
\{E:
E\text{ is a \(\Theta\)-class separating }x\text{ and }y\}
\bigr|\\
&=d_G(x,y).
\end{align*}
Thus, \(\pi\) is an isometric embedding of \(G\) into a Cartesian product of \(k\) trees. Hence, $tdim(G)\leq k=\chi(G^\#).$

Combining the two inequalities completes the proof.
\end{proof}

Now, we consider median graphs. Recall the standard result on median graphs that will be used below.

\begin{lemma} \label{thm:Bandelt-retract}\cite{Bandelt1984} Let \(H\) be a connected induced subgraph of a median graph \(M\) with \(|V(H)|\ge 2\). Then \(H\) is a retract of \(M\) if and only if \(H\) is a median subgraph of \(M\).
\end{lemma}

By Lemma \ref{thm:Bandelt-retract}, we obtain the following corollary immediately.
\begin{corollary}\label{cor:median-retract}
    Let $M$ be a median graph and $H$ an isometric subgraph of $M$. If $H$ is median, then it is a retract of $M$.
\end{corollary}

Recall that the isometric dimension of a partial cube $G$ is denoted by $idim(G)$. Crawford and Ir\v{s}i\v{c} Chenoweth \cite{CrawfordIrsicChenoweth2025} obtained that $c(M)\le \left\lceil (idim(M)+1)/2 \right\rceil$ for a median graph $M$. 
The following result replaces the isometric dimension by the generally smaller tree-dimension.

\begin{theorem}
\label{thm:cop-bound-from-tdim}
Let $M$ be a finite median graph. Then
\[
c(M)\le \left\lceil (tdim(M)+1)/2 \right\rceil=\lceil (\chi(M^{\#})+1)/2 \rceil.
\]
\end{theorem}

\begin{proof}
When $M=K_1$, $c(M)=1$, $tdim(M)=1$, $\chi(M^{\#})=0$. Then $c(M)= \left\lceil (tdim(M)+1)/2 \right\rceil$ $=\lceil (\chi(M^{\#})+1)/2 \rceil=1$. The inequality holds.

Now, assume $M\neq K_1$, i.e., $|V(M)|\geq 2$. By the definition of \(t=tdim(M)\), there exist nontrivial trees \(T_1,\ldots,T_t\) and an isometric embedding
\[
\phi\colon M\longrightarrow
P:=T_1\square\cdots\square T_t.
\]
Since each \(T_i\) is a tree, \(P\) is a median graph. $\phi(M)$ is isomorphic to $M$, so it is median. By Corollary \ref{cor:median-retract}, 
\(\phi(M)\) is a retract of \(P\). Therefore, by Lemma \ref{lem:retract-cop}, $c(M)=c(\phi(M))\le c(P).$ By Lemma \ref{thm:MM},
\[ c(P) = \left\lceil (t+1)/2\right\rceil . \] 

Combined with Theorem \ref{thm:tree-dimension-crossing}, 
\[ c(M) \le c(P) =\left\lceil (t+1)/2\right\rceil = \lceil (\chi(M^\#)+1)/2\rceil . \]

This proves the theorem.
\end{proof}

The improvement in Theorem \ref{thm:cop-bound-from-tdim} can be arbitrarily large. If \(T\) is a tree with \(m\) edges, then every edge of \(T\) forms a distinct \(\Theta\)-class, so the smallest hypercube into which \(T\) embeds isometrically has dimension \(m\). In contrast, \(tdim(T)=1\). The previous estimate therefore gives \(\lceil(m+1)/2\rceil\), whereas our bound gives \(c(T)\leq 1\), which is exact. On the other hand, \(tdim(Q_d)=d\), so the new bound retains equality for hypercubes. Thus, the tree-dimension formulation strengthens the general estimate for median graphs without losing sharpness for the extremal examples given by hypercubes.

The remainder of this section is devoted to establishing bounds on the cop numbers of simplex graphs. To obtain the lower bound for the cop number of a simplex graph, we first show that every clique of the base graph induces a retract that is a hypercube.

\begin{lemma}\label{lem:clique-hypercube-retract}
Let \(G\) be a finite graph, and let \(K\) be a nonempty clique of
\(G\). Then the family $\{A:A\subseteq K\}$ induces a subgraph of \(S(G)\) isomorphic to \(Q_{|K|}\), and this induced hypercube is a retract of \(S(G)\).
\end{lemma}

\begin{proof}
Every subset of \(K\) is a clique of \(G\), and two subsets of \(K\)
are adjacent in \(S(G)\) exactly when they differ in one element. Hence, these subsets induce a hypercube isomorphic to \(Q_{|K|}\).

Set $|V(G)|=n$. Since both $S(G)$ and $Q_{|K|}$ are isometric subgraph of $Q_n$, for any $a,b\in V(Q_{|K|})$, $d_{Q_{|K|}}(a,b)=d_{Q_n}(a,b)=d_{S(G)}(a,b)$. Thus, $Q_{|K|}$ is an isometric subgraph of $S(G)$. Combined with the fact that hypercubes are median graphs,  by Corollary \ref{cor:median-retract}, \(Q_{|K|}\) is a retract of \(S(G)\).
\end{proof}

The following standard identity was proved by Klav\v{z}ar and Mulder \cite{KlavzarMulder2002}.

\begin{lemma}\label{lem:simplex-crossing}\cite{KlavzarMulder2002}
For every finite graph \(G\),
\[
        (S(G))^\#\cong G.
\]
\end{lemma}


The following theorem provides a general starting point for studying cop numbers of simplex graphs.

\begin{theorem}\label{thm:main-simplex}
Let \(G\) be a finite graph. Then
\[
 \left\lceil(\omega(G)+1)/2\right\rceil
 \leq c(S(G))
 \leq
 \left\lceil(\chi(G)+1)/2\right\rceil.
\]
In particular, if \(G\) is perfect, then
\[
 c(S(G))
 =
 \left\lceil(\omega(G)+1)/2\right\rceil
 =
 \left\lceil(\chi(G)+1)/2\right\rceil.
\]
\end{theorem}

\begin{proof}
If \(V(G)=\varnothing\), then \(S(G)=K_1\) and
\[
c(S(G))=1
=
\left\lceil(\omega(G)+1)/2\right\rceil
=
\left\lceil(\chi(G)+1)/2\right\rceil.
\]
Hence, assume that \(V(G)\neq\varnothing\).

Let \(K\) be a maximum clique of \(G\). By Lemma~\ref{lem:clique-hypercube-retract}, \(S(G)\) contains a retract isomorphic to \(Q_{\omega(G)}\). Hence,
\[
 c(S(G))
 \geq c(Q_{\omega(G)})
 =
 \left\lceil(\omega(G)+1)/2\right\rceil.
\]

On the other hand, Lemma~\ref{lem:simplex-median} shows that \(S(G)\) is a finite median graph, and Lemma~\ref{lem:simplex-crossing} gives $(S(G))^\#\cong G.$ Applying Theorems \ref{thm:tree-dimension-crossing} and \ref{thm:cop-bound-from-tdim}, we obtain
\[
\begin{aligned}
 c(S(G))
 &\leq
 \lceil
 (\chi((S(G))^\#)+1)/2
 \rceil\\
 &=
 \left\lceil
 (\chi(G)+1)/2
 \right\rceil.
\end{aligned}
\]

If \(G\) is perfect, then \(\chi(G)=\omega(G)\). Therefore,
\[
 c(S(G))
 =
 \left\lceil(\omega(G)+1)/2\right\rceil
 =
 \left\lceil(\chi(G)+1)/2\right\rceil.
\]
\end{proof}


Every subset of \(V(K_n)\) is a clique of \(K_n\). Hence, $S(K_n)\cong Q_n.$ Moreover, \(K_n\) is perfect and
\[
\omega(K_n)=\chi(K_n)=n.
\]
The following result therefore follows immediately from Theorem~\ref{thm:main-simplex}. For every integer \(n\geq 1\),
\[
c(Q_n)=\left\lceil(n+1)/2\right\rceil.
\]
This is the simplex-graph specialization of the classical formula of Maamoun and Meyniel \cite{MaamounMeyniel1987} for the cop number of a hypercube.

\section{The cop number of bipartite wheels and Fibonacci cubes}

In this section, we determine the exact cop number of bipartite wheels and Fibonacci cubes.

For every integer \(n\geq3\), $S(C_n)\cong BW_n$. Applying Theorem~\ref{thm:main-simplex}, we obtain the following exact
result for bipartite wheels.

\begin{theorem}\label{cor:bipartite-wheel}
For every integer \(n\geq3\),
\[
c(\mathrm{BW}_n)=2.
\]
\end{theorem}

\begin{proof}
Since
\[
\omega(C_n)=
\begin{cases}
2, & \text{if \(n\ge4\)},~~~~\\
3, & \text{if \(n=3\)};
\end{cases}
\]
and
\[
\chi(C_n)=
\begin{cases}
2, & \text{if \(n\) is even},\\
3, & \text{if \(n\) is odd}.
\end{cases}
\]
Theorem~\ref{thm:main-simplex} gives
\[
2
=
\left\lceil(\omega(C_n)+1)/2\right\rceil
\leq c(S(C_n))
\leq
\left\lceil(\chi(C_n)+1)/2\right\rceil
=2.
\]
Therefore,
\[
c(\mathrm{BW}_n)=c(S(C_n))=2.
\]
\end{proof}

Let \(P_n\) be the path on vertex set \(\{1,\ldots,n\}\). A clique of \(\overline{P_n}\) is precisely an independent set of \(P_n\). Hence, the characteristic vectors of the cliques of \(\overline{P_n}\) are exactly the binary strings of length \(n\) containing no two consecutive \(1\)'s. Therefore,
\[
\Gamma_n\cong S(\overline{P_n}).
\]
For \(n=0\), we take \(P_0\) to be the empty graph, so that $S(\overline{P_0})=K_1=\Gamma_0.$

The following standard parameters of the complement of a path have been known:
\[
        \chi(\overline{P_n})
        =
        \omega(\overline{P_n})
        =
        \alpha(P_n)
        =
        \left\lceil n/2\right\rceil.
\]

\begin{theorem}\label{thm:fibonacci-main}
For every \(n\geq0\),
\[
c(\Gamma_n)
=
\left\lceil
(\lceil n/2\rceil+1)/2
\right\rceil.
\]
Equivalently, for every integer \(q\geq0\),
\[
\begin{aligned}
c(\Gamma_{4q})   =q+1,
c(\Gamma_{4q+1}) =q+1,
c(\Gamma_{4q+2}) =q+1,
c(\Gamma_{4q+3}) =q+2.
\end{aligned}
\]
\end{theorem}

\begin{proof}
By the simplex-graph representation
\[
\Gamma_n\cong S(\overline{P_n}),
\]
Theorem~\ref{thm:main-simplex} gives
\[
\left\lceil
(\omega(\overline{P_n})+1)/2
\right\rceil
\leq c(\Gamma_n)
\leq
\left\lceil
(\chi(\overline{P_n})+1)/2
\right\rceil.
\]
Since
\[
\omega(\overline{P_n})
=
\chi(\overline{P_n})
=
\left\lceil n/2\right\rceil,
\]
the two bounds coincide, and hence
\[
c(\Gamma_n)
=
\left\lceil
(\lceil n/2\rceil+1)/2
\right\rceil.
\]
The four residue-class formulas follow by direct evaluation.
\end{proof}

Table \ref{tab:fibonacci_cop_transposed} lists the cop numbers of Fibonacci cubes $\Gamma_{n}$ for $n\in\{0,1,\ldots, 10\}$. The values agree with the exact formula given in Theorem \ref{thm:fibonacci-main} and illustrate the four-periodic pattern modulo 4.

\begin{table}[htbp]
\centering
\caption{The cop numbers of small Fibonacci cubes.}
\label{tab:fibonacci_cop_transposed}
\begin{tabular}{c|*{10}{c|}c}
\hline\hline
$n$ & 0 & 1 & 2 & 3 & 4 & 5 & 6 & 7 & 8 & 9 & 10  \\
\hline
$|V(\Gamma_n)|$ & 1 & 2 & 3 & 5 & 8 & 13 & 21 & 34 & 55 & 89 & 144  \\
\hline
$c(\Gamma_n)$ & 1 & 1 & 1 & 2 & 2 & 2 & 2 & 3 & 3 & 3 & 3  \\
\hline\hline
\end{tabular}
\end{table}

We close this section with a remark on a previously stated upper bound for Fibonacci cubes.

\begin{remark}
Crawford and Ir\v{s}i\v{c} Chenoweth~\cite{CrawfordIrsicChenoweth2025} stated that 
\(
c(\Gamma_n)\leq \left\lceil n/4\right\rceil
\) 
for all \(n\geq9\).
This statement is false, because, for \(q\geq2\),
\(
c(\Gamma_{4q+3})
=
q+2
>
q+1
=
\left\lceil(4q+3)/4\right\rceil.
\)
\end{remark}

\section{The cop numbers of Lucas cubes}

In this section, we determine the exact cop number of Lucas cubes.
\begin{theorem}\label{thm:lucas-main}
For every \(n\geq0\),
\[
 c(\Lambda_n)
 =
 \left\lceil
 (\lfloor n/2\rfloor+1)/2
 \right\rceil=\lfloor n/4\rfloor+1.
\]
Equivalently, for every integer \(q\geq0\),
\[
 c(\Lambda_{4q})
 =
 c(\Lambda_{4q+1})
 =
 c(\Lambda_{4q+2})
 =
 c(\Lambda_{4q+3})
 =
 q+1.
\]
\end{theorem}

\begin{proof}
First, we treat the smaller values separately. By convention, $\Lambda_0=\Lambda_1=K_1,$  while $\Lambda_2\cong P_3
~\text{and}~\Lambda_3\cong K_{1,3}.$ Thus,
\[
c(\Lambda_n)=1
~\text{for }0\leq n\leq3.
\]

Now, assume that \(n\geq4\). Since $\Lambda_n\cong S(\overline{C_n})$, Theorem~\ref{thm:main-simplex} gives
\[
 \left\lceil
 (\omega(\overline{C_n})+1)/2
 \right\rceil
 \leq c(\Lambda_n)
 \leq
 \left\lceil
 (\chi(\overline{C_n})+1)/2
 \right\rceil.
\]
Since $\omega(\overline{C_n})
=
\left\lfloor n/2\right\rfloor
~\text{and}~
\chi(\overline{C_n})
=
\left\lceil n/2\right\rceil
~\text{for}~n\geq4,$ we obtain
\begin{equation}\label{eq:lower-and-upper-bounds-of-Lambda}
 \left\lceil(\lfloor n/2\rfloor+1)/2\right\rceil
 \leq c(\Lambda_n)
 \leq
 \left\lceil(\lceil n/2\rceil+1)/2\right\rceil.    
\end{equation}
For \(n\equiv0,1,2\pmod4\), the two bounds coincide. So we only need to prove: for $q\geq 1$,

\begin{equation}\label{eq:Lambda-4q+3}
    c(\Lambda_{4q+3})=q+1.
\end{equation}


By \eqref{eq:lower-and-upper-bounds-of-Lambda}, 
\begin{equation}\label{eq:lower-bounds}
    c(\Lambda_{4q+3})\geq \left\lceil[\lfloor (4q+3)/2\rfloor+1]/2\right\rceil=q+1
\end{equation}
Now, we prove $c(\Lambda_{4q+3})\leq q+1.$

Put $n=4q+3~\text{and}~k=q+1.$ We use the simplex-graph representation $\Lambda_n\cong S(\overline{C_n}).$ Thus, a vertex of \(\Lambda_n\) will be regarded as a clique of \(\overline{C_n}\). A move consists of adding or deleting one vertex, provided that the resulting set remains a clique of \(\overline{C_n}\).

Partition the cyclic coordinate set $V(C_n)=\{1,2,\ldots,n\}$ into \(k=q+1\) consecutive blocks $B_0=\{1,2,3\}$ and
\[
        B_i=\{4i,4i+1,4i+2,4i+3\},
        ~ 1\leq i\leq q.
\]
The vertices in each \(B_i\) occur consecutively on \(C_n\). In particular, \(C_n[B_0]\cong P_3\), while \(C_n[B_i]\cong P_4\) for \(1\leq i\leq q\). Therefore, every clique of \(\overline{C_n}\), equivalently every independent set of \(C_n\), contains at most two vertices from each block \(B_i\).

We use \(k=q+1\) cops, denoted by $D_0,D_1,\ldots,D_q,$ where cop \(D_i\) is assigned to the block \(B_i\). Initially, all cops are placed at the empty clique \(\varnothing\). For a clique \(A\) of \(\overline{C_n}\), define
\[
        \pi_i(A)=A\setminus B_i.
\]
Since every subset of a clique is again a clique, \(\pi_i(A)\) is a vertex of \(S(\overline{C_n})\).

The strategy of cop \(D_i\) is defined as follows. Suppose that the robber is currently at a clique \(R\), and it is the cops' turn to move. 
\begin{enumerate}
    \item If the current position of \(D_i\) is different from \(\pi_i(R)\), then \(D_i\) moves one step along a shortest path in \(S(\overline{C_n})\) from its current position toward \(\pi_i(R)\).
    \item If \(D_i\) is already at \(\pi_i(R)\) for some $i$, since $|R\cap B_i|\leq 2$, divide it  into the following three cases:
\begin{enumerate}[label=(\roman*)]
\item If \(R\cap B_i=\varnothing\), then $\pi_i(R)=R,$ so \(D_i\) already occupies the robber's position and the robber is captured.

\item If \(|R\cap B_i|=1\), then \(D_i\) adds the unique vertex of \(R\cap B_i\). The resulting clique is exactly \(R\), so the robber is captured.

\item If \(|R\cap B_i|=2\), then \(D_i\) remains at \(\pi_i(R)\).
\end{enumerate}
\end{enumerate}

We prove that this strategy guarantees capture. Suppose, to the contrary, that the robber can evade the cops indefinitely. We now fix the time notation. Let $R_{0}$ and $D_{i,0}$ denote the initial positions of robber and cop $D_{i}$, respectively, so $D_{i,0}=\varnothing$ for all $i$. For \(t\geq 1\), let \(R_t\) and \(D_{i,t}\) denote the positions of the robber and cop \(D_i\), respectively, immediately after the cops have completed their \(t\)-th response. The robber then moves from \(R_t\) to a clique \(\widehat R_t\), after which the cops make their next response. Since the robber does not move during the cops' response, we have $R_{t+1}=\widehat R_t.$
For every \(i\in\{0,1,\ldots,q\}\), define
\[
        d_i(t)
        =
        d_{S(\overline{C_n})}
        \bigl(D_{i,t},\pi_i(R_t)\bigr)
        =
        \bigl|D_{i,t}\triangle\pi_i(R_t)\bigr|.
\]
We first show that

{\bf Claim} For every \(i\) and every \(t\), 
\[
        d_i(t+1)\leq d_i(t),
\]
unless the robber is captured during the cops' response.

Suppose first that the robber remains at \(R_t\). Then, none of the targets \(\pi_i(R_t)\) changes. Therefore, if \(d_i(t)>0\), cop \(D_i\) moves one step closer to the same target, and therefore $d_i(t+1)=d_i(t)-1.$ If \(d_i(t)=0\), then \(D_{i,t}=\pi_i(R_t)\); the strategy either captures the robber or leaves \(D_i\) at the target. Hence, $d_i(t+1)=d_i(t)=0.$

Now suppose that the robber changes one coordinate, and let \(B_j\) be the block containing this coordinate.

If \(i=j\), deleting all coordinates of \(B_j\) eliminates the changed coordinate, and therefore $\pi_j(\widehat R_t)=\pi_j(R_t).$ Thus, the target of \(D_j\) does not change. If \(d_j(t)>0\), then \(D_j\) moves one step closer to this same target, so $d_j(t+1)=d_j(t)-1.$ If \(d_j(t)=0\), then \(D_{j,t}\) already occupies the unchanged target. Unless the robber is captured, \(D_j\) remains at this target, and hence
$d_j(t+1)=0.$

Otherwise, if \(i\neq j\), i.e,, the coordinate changed by the robber lies outside \(B_i\), the two targets \(\pi_i(R_t)\) and
\(\pi_i(\widehat R_t)\) differ in exactly this coordinate, and in particular $d_{S(\overline{C_n})}\bigl(\pi_i(R_t),\pi_i(\widehat R_t)\bigr)=1.$
By the triangle inequality,
\[
\begin{aligned}
d_{S(\overline{C_n})}
\bigl(D_{i,t},\pi_i(\widehat R_t)\bigr)
&\leq
d_{S(\overline{C_n})}
\bigl(D_{i,t},\pi_i(R_t)\bigr)+1\\
&=d_i(t)+1.
\end{aligned}
\]
If \(D_{i,t}\neq\pi_i(\widehat R_t)\), cop \(D_i\) moves one step toward the new target. Hence, $d_i(t+1)\leq d_i(t).$ If \(D_i(t)=\pi_i(\widehat R_t)\), then \(d_i(t+1)=0\), unless capture has already occurred. Thus, in all cases, $d_i(t+1)\leq d_i(t).$ The proof of the claim is completed. 

It follows that every sequence $d_i(0),d_i(1),d_i(2),\ldots$ is a non-increasing sequence of nonnegative integers. Hence, there exists a time \(T\) such that the vector $\bigl(d_0(t),d_1(t),\ldots,d_q(t)\bigr)$ is constant for all \(t\geq T\). We refer to this period as the stable phase.

We first show that the robber cannot remain stationary during the stable phase. Suppose that the robber remains at \(R_t\) for some \(t\geq T\). If \(d_i(t)>0\) for some \(i\), then \(D_i\) moves one step closer to the unchanged target, giving $d_i(t+1)=d_i(t)-1,$ contrary to stability. Thus, stability would require $d_i(t)=0$ for every \(i\). Therefore, $D_{i,t}=\pi_i(R_t)$ for every \(i\).

Since the robber is not captured, the first two cases of the strategy cannot occur. Hence, $|R_t\cap B_i|=2$ for every \(i\in\{0,1,\ldots,q\}\). Because the blocks form a partition
of \(V(C_n)\), this gives
\[
        |R_t|
        =
        \sum_{i=0}^{q}|R_t\cap B_i|
        =
        2(q+1)
        =
        2q+2.
\]
However, \(R_t\) is a clique of \(\overline{C_{4q+3}}\), and
\[
        \omega(\overline{C_{4q+3}})
        =
        \alpha(C_{4q+3})
        =
        \left\lfloor(4q+3)/2\right\rfloor
        =
        2q+1.
\]
This is a contradiction. Therefore, the robber cannot remain stationary during the stable phase.

We next show that, during the stable phase, the robber cannot change a coordinate belonging to a block \(B_j\) for which \(d_j(t)>0\). In fact, such a move leaves the target \(\pi_j(R_t)\) unchanged, and therefore cop \(D_j\) moves one step closer to that target. This gives $d_j(t+1)=d_j(t)-1,$ again contradicting stability.

Consequently, every robber move during the stable phase changes one coordinate in some block \(B_j\) satisfying $d_j(t)=0.$ For this block, $D_{j,t}=\pi_j(R_t).$ Since the robber changes a coordinate belonging to \(B_j\), the
projection outside \(B_j\) is unchanged. Hence,
\[
        \pi_j(\widehat R_t)=\pi_j(R_t)=D_{j,t}.
\]

Let $s=|R_t\cap B_j|.$ Since a clique of \(\overline{C_n}\) contains at most two vertices from \(B_j\), we have $s\in\{0,1,2\}.$  If \(s=0\), that is, $R_t=\pi_j(R_t)$, then the robber has been captured. If \(s=2\), then the robber cannot add a third vertex from \(B_j\), because a clique of \(\overline{C_n}\) contains at most two vertices from this block. Hence, the robber must delete one of the two vertices, and $|\widehat R_t\cap B_j|=1.$ Again, cop \(D_j\) adds the unique remaining vertex of \(\widehat R_t\cap B_j\) and captures the robber.

Thus, suppose that \(s=1\). If the robber deletes the unique vertex of \(R_t\cap B_j\), then
\[
        \widehat R_t=\pi_j(\widehat R_t)=D_{j,t},
\]
so the robber is captured immediately. Therefore, the only move in \(B_j\) that does not lead to immediate capture is to add a second vertex. In that case, $|\widehat R_t\cap B_j|=2
~\text{and}~|\widehat R_t|=|R_t|+1.$

We have proved that every robber move during the stable phase that avoids capture strictly increases the size of the robber's clique. Thus, $|R_{t+1}|=|R_t|+1$ for every \(t\geq T\), as long as capture does not occur. This cannot continue indefinitely, because every robber position is a clique of \(\overline{C_{4q+3}}\) with the size at most $\omega(\overline{C_{4q+3}})=2q+1.$ This final contradiction shows that the robber cannot evade capture forever.

Therefore, \(q+1\) cops have a winning strategy on \(\Lambda_{4q+3}\), and hence
\begin{equation}\label{eq:upper-bounds}
    c(\Lambda_{4q+3})\leq q+1.
\end{equation}
        
Combined with \eqref{eq:lower-bounds} and \eqref{eq:upper-bounds}, \eqref{eq:Lambda-4q+3} holds.

The proof of this theorem is completed.
\end{proof}

Table \ref{tab:lucas_cop_transposed} lists the cop numbers of Lucas cubes $\Lambda_n$ for $n\in\{0,1,\ldots,10\}$. The values agree with the exact formula in Theorem~\ref{thm:lucas-main}. In particular, the previously delicate residue class $n\equiv 3 \pmod 4$ also follows from the same formula.

\begin{table}[htbp]
\centering
\caption{The cop numbers of small Lucas cubes.}
\label{tab:lucas_cop_transposed}
\begin{tabular}{c|*{10}{c|}c}
\hline\hline
$n$ & 0 & 1 & 2 & 3 & 4 & 5 & 6 & 7 & 8 & 9 & 10 \\
\hline
$|V(\Lambda_n)|$& 1 & 1 & 3 & 4 & 7 & 11 & 18 & 29 & 47 & 76 & 123  \\
\hline
$c(\Lambda_n)$ & 1 & 1 & 1 & 1 & 2 & 2 & 2 & 2 & 3 & 3 & 3  \\
\hline\hline
\end{tabular}
\end{table}

\begin{remark}
   Crawford and Ir\v{s}i\v{c} Chenoweth~\cite{CrawfordIrsicChenoweth2025} stated that
\(
c(\Lambda_n)\leq \left\lceil n/4\right\rceil
\)  
when $n\geq 9$. This statement is false. By Theorem \ref{thm:lucas-main}, When $n=4q$ and $q\geq 3$ is an integer, $c(\Lambda_n)= \left\lfloor n/4\right\rfloor+1=q+1>q=\left\lceil n/4\right\rceil$.
\end{remark}

\section{Conclusion}

We first established an improved general upper bound for finite median graphs. Our result refines the previous bound and can be strictly stronger, while remaining sharp for hypercubes. We then established general lower and upper bounds for the cop number of a simplex graph:
\[
        \left\lceil(\omega(G)+1)/2\right\rceil
        \leq c(S(G)) \leq
        \left\lceil(\chi(G)+1)/2\right\rceil .
\]
By leveraging these bounds, we have established the exact values of cop numbers of hypercubes, bipartite wheels, and Fibonacci cubes; furthermore, through a detailed structural analysis, we have also determined the cop numbers of Lucas cubes.

The exact formulas for Fibonacci and Lucas cubes can be viewed from a broader downset perspective. Both families are induced subgraphs of hypercubes whose vertex sets are closed downward under the coordinatewise order. This suggests that the hypercube-retract lower bound used throughout the paper may remain sharp for a larger class of partial cubes.

Daisy cubes provide a natural setting for this question. Let $\le$ be a \emph{partial order} on $\{0,1\}^n$ defined in the following way : $u_1u_2\ldots u_n\le v_1v_2\ldots v_n$ if $u_i\le v_i$ is true for all $i\in \{1,2,\ldots, n\}$. For
\(X\subseteq \{0,1\}^n\), the \emph{daisy cube} \(Q_n(X)\) is the subgraph of \(Q_n\) induced by
\[
        \{u\in \{0,1\}^n : u\leq x \text{ for some } x\in X\}.
\]
For a binary string
\(u=u_1u_2\cdots u_n\), let $w(u)=|\{i : u_i=1\}|$ be its \emph{Hamming weight}, and set
\[
        \rho(Q_n(X))=
        \max\{w(u):u\in V(Q_n(X))\}.
\]

If \(\rho(Q_n(X))=0\), then \(Q_n(X)=K_1\), and  $c(Q_n(X))=\left\lceil(\rho(Q_n(X))+1)/2\right\rceil=1$. Hence, assume that \(\rho(Q_n(X))\geq 1\). Choose a vertex $x=x_1x_2\cdots x_n\in V(Q_n(X))$ of weight \(\rho(Q_n(X))\), and let
\[
K=\{i\in\{1,\ldots,n\}:x_i=1\}.
\]
Then \(|K|=\rho(Q_n(X))\). Since \(Q_n(X)\) is downward closed, the vertices whose \(1\)-coordinates form subsets of \(K\) induce a subcube isomorphic to \(Q_{|K|}\). We can construct a retraction of $Q_n(X)$ onto \(Q_{|K|}\) in the following.

Let $u=u_1u_2\cdots u_n$, $v=v_1v_2\cdots v_n$ be two binary strings of lengths $n$. $u\wedge v$ and $u\setminus v$ is defined as follows, respectively:
$$
(u\wedge v)_i=\begin{cases}
    1, & \mbox{if }u_i=1\mbox{ and }v_i=1;\\
    0, & \mbox{otherwise,}
\end{cases}
$$
$$
(u\setminus v)_i=\begin{cases}
    1, & \mbox{if }u_i=1\mbox{ and }v_i=0;\\
    0, & \mbox{otherwise.}
\end{cases}
$$
for all $1\leq i\leq n$. For some $1\leq i\leq n$, Denote \(v^{(i)}\) the binary string obtained from \(v\) by changing the \(i\)-th coordinate.

Now, fix a coordinate \(k_0\in K\). For every $v\in V(Q_n(X))$, define
\[
r(v)=
\begin{cases}
v\wedge x,
   & \text{if } w(v\setminus x) \text{ is even},\\[2mm]
(v\wedge x)^{(k_0)},
   & \text{if } w(v\setminus x) \text{ is odd}.
\end{cases}
\]

By the definition, we can see that for any $v\in V(Q_n(X))$, $r(v)\leq x$, i.e., $r(v)\in V(Q_{|K|})$. Moreover, if $r(v)\in V(Q_{|K|})$, i.e., $v\leq x$, then $w(v\setminus x)=w(0^n)=0$, so $r(v)=v\wedge x=v$. 

Let \(u\) and \(v\) be adjacent vertices of \(Q_n(X)\). Then $u$ and $v$ are different only at the $i_0$-th coordinate. If $i_0\in K$, $u\setminus x=v\setminus x$.  
By the definition, Then $r(u)$ and $r(v)$ are also different only at the $i_0$-th coordinate, that is, \(r(u)\) and \(r(v)\) are adjacent in \(Q_{|K|}\). If \(i_0 \notin K\). Then $u\wedge x=v\wedge x$, while \(|u\setminus x|\) and \(|v\setminus x|\) have opposite parities. By the definition, Then $r(u)$ and $r(v)$ are different only at the $k_0$-th coordinate, 
and again \(r(u)\) and \(r(v)\) are adjacent in \(Q_{|K|}\). By the above analysis, \(r\) is a retraction of \(Q_n(X)\) onto \(Q_{|K|}\).   
Therefore,
\[
c(Q_n(X))
\geq
c\bigl(Q_{\rho(Q_n(X))}\bigr)
=
\left\lceil
(\rho(Q_n(X))+1)/2
\right\rceil.
\]

The formulas obtained in this paper for Fibonacci and Lucas cubes show that this lower bound is sharp for two important families of daisy cubes. This leads to the following problem.

\begin{problem}\label{problem:daisy cube}
Is it true that for every daisy cube \(Q_n(X)\),
\[
        c(Q_n(X))=
        \left\lceil(\rho(Q_n(X))+1)/2\right\rceil?
\]
\end{problem}


The tree-dimension approach also suggests a broader problem for partial cubes. Theorem~\ref{thm:cop-bound-from-tdim} shows that
\[
c(G)\leq
\left\lceil
(tdim(G)+1)/2
\right\rceil
\]
when \(G\) is median. The median assumption is essential to our proof of Theorem \ref{thm:cop-bound-from-tdim}, since it ensures that an isometric copy of $G$ in a Cartesian product of trees is a retract of that product. An arbitrary partial cube need not have this retraction property. Theorem \ref{thm:tree-dimension-crossing} prove that $tdim(G)=\chi(G^\#)$ holds for every finite nontrivial partial cube \(G\), it is nevertheless natural to ask whether the resulting cop-number bound remains valid without the median assumption.

\begin{problem}\label{conj:partial-cube-tree-dimension}
Does for every finite partial cube \(G\),
\[
c(G)\leq \left\lceil (tdim(G)+1)/2 \right\rceil = \left\lceil (\chi(G^\#)+1)/2 \right\rceil ?
\]
\end{problem}

This problem strengthens the previously proposed upper bound in terms of the isometric dimension. It would also be best possible, as equality holds for hypercubes and, more generally, for the simplex graphs of perfect graphs. We are currently investigating the problem \ref{problem:daisy cube}, and we plan to report our results in subsequent papers.

\section{Declaration of competing interest}
The authors declare that they have no known competing financial interests or personal relationships that could have appeared to influence the work reported in this paper.

\section{Data availability}
No data was used for the research described in the article.

\section{Acknowledgments}
This work was partially supported by the National Natural Science Foundation of China (No. 12071194).

\end{CJK}
\end{document}